\documentclass[12pt]{article}
\usepackage{amssymb,amsmath,amstext, amsthm}
\usepackage{amsmath,amssymb,amscd,amsmath,amstext,mathrsfs,amsfonts,pb-diagram}%
\newtheorem{theorem}{Theorem}
\theoremstyle{plain}

\newtheorem{corollary}{Corollary}

\newtheorem{definition}{Definition}
\newtheorem{example}{Example}

\newtheorem{lemma}{Lemma}

\newtheorem{problem}{Problem}
\newtheorem{proposition}{Proposition}

\numberwithin{equation}{section}

\newcommand{\Pn}{\P(\C^{n+1})}

\def\x{\dot{x}}

\def\u{\dot{u}}

\def\sic{\si_n^2}

\def\al{\alpha}

\def\ga{\gamma}

\def\de{\delta}

\def\ze{\zeta}

\def\ka{\kappa}
\def\la{\lambda}

\def\si{\sigma}

\def\Si{\Sigma}

\def\C{{\mathbb C}}

\def\R{{\mathbb R}}
\def\K{{\mathbb K}}

\def\P{{\mathbb P}}
\def\S{{\mathbb S}}

\def\CM{\mathcal{M}}
\def\CN{\mathcal{N}}
\def\CU{\mathcal{U}}

\def\CW{\mathcal{W}}

\def\GL{{\mathbb G}{\mathbb L}_{n,m}}
\def\GLS{{\mathbb G }{\mathbb L}_{n,m}^>}
\def\Knm{{\mathbb K}^{n \times m}}

\def\PKn{\P\left(\K^{n \times (n+1)}\right)}
\def\Pn{\P_n\left(\K\right)}

\newcommand{\rank}{\mbox{rank }}
\newcommand{\trace}{\mbox{trace }}

\newcommand{\diag}{\mbox{diag }}
\newcommand{\ra}{\rightarrow}

\title{Convexity properties of the condition number.
\thanks{Mathematics Subject Classification (MSC2000): 65F35 (Primary),
15A12 (Secondary).}}

\author{Carlos Beltr\'an
        \thanks{C. Beltr\'an,
               Departmento de Matem\'aticas, Estad\'isticas y Computac\'ion
               Universidad de Cant\'abria,
               Santander, Espa\~na
               ({\tt beltranc@gmail.com}). CB was supported by MTM2007-62799 and by a Spanish postdoctoral grant. 
}
\and      Jean-Pierre Dedieu
          \thanks{J.-P. Dedieu,
               Institut de Math\'ematiques,
               Universit\'e Paul Sabatier,
               31062 Toulouse cedex 09, France
                ({\tt jean-pierre.dedieu@math.univ-toulouse.fr}). J.-P. Dedieu was supported by the ANR Gecko.
}
\and    Gregorio Malajovich
        \thanks{G. Malajovich,
                Departamento de Matem\'atica Aplicada,
                Universidade Federal de Rio de Janeiro,
                Caixa Postal 68530,
                CEP 21945-970, Rio de Janeiro, RJ, Brazil
                ({\tt gregorio@ufrj.br}). He was partially supported by CNPq 
(Conselho Nacional de Desenvolvimento Cient\'i{\i}fico e Tecnol\'ogico - Brasil), 
by FAPERJ
  (Funda\c{c}$\tilde{a}$o Carlos Chagas de Amparo \`a Pesquisa do Estado do Rio de
Janeiro) and by the
  Brazil-France agreement of cooperation in Mathematics.
}
\and    Mike Shub
				\thanks{M. Shub,
               Department of Mathematics, 
               University of Toronto, Toronto, 
               Ontario, Canada M5S 2E4
               ({\tt shub.michael@gmail.com}). CB and MS were supported by an NSERC Discovery Grant.
}               
}

\begin{document} 
\maketitle
\begin{abstract}
We define in the space of $n \times m$ matrices of rank $n$, $n\leq m$, the condition Riemannian structure as follows: For a given matrix $A$ the tangent space at $A$ is equipped with the Hermitian inner product obtained by multiplying the usual Frobenius inner product by the inverse of the square of the smallest singular value of $A$ denoted $\si_n(A)$.
When this smallest singular value has multiplicity $1$, the function $A \rightarrow \log (\si_n(A)^{-2})$ is a convex function with respect to the condition Riemannian structure that is $t \rightarrow \log (\si_n(A(t))^{-2})$ is convex, in the usual sense for any geodesic $A(t)$.
In a more abstract setting, a function $\alpha$ defined on a Riemannian manifold $(\CM , \left\langle ,\right\rangle)$ is said to be self-convex when $\log \alpha (\gamma(t))$ is convex for any geodesic in $(\CM , \al \left\langle , \right\rangle)$. Necessary and sufficient conditions for self-convexity are given when $\alpha$ is $C^2$. When $\alpha(x) = d(x,\CN)^{-2}$ where $d(x,\CN)$ is the distance from $x$ to a $C^2$ submanifold  $\CN \subset \R^j$ we prove that $\alpha$ is self-convex when restricted to the largest open set of  points $x$ where there is a unique closest point in $\CN$ to $x$. We also show,  using this more general notion, that the square of the condition number $\left\| A \right\|_F/ \sigma_n(A)$ is self-convex in projective space and the solution variety.
\end{abstract}


\section{Introduction}

Let two integers $1\leq n\leq m$ be given and let us consider the space of matrices $\Knm$, $\K = \R$ or $\C$, equipped with the Frobenius Hermitian product
$$\left\langle M,N \right\rangle_F = \trace (N^*M) = \sum_{i,j}m_{ij}\overline{n_{ij}}.$$
Given an absolutly continuous path $A(t)$, $a \le t \le b$, its length is given by the integral
$$L = \int_a^b \left\|\frac{dA(t)}{dt}\right\|_F dt,$$
and the shortest path connecting $A(a)$ to $A(b)$ is the segment connecting them. Consider now the problem of connecting these two matrices with the shortest possible path in staying, as much as possible, away from the set of ``singular matrices'' that is the matrices with non-maximal rank. 

The singular values of a matrix $A \in \Knm$ are denoted in non-increasing order: 
$$\si_1(A) \ge \ldots \ge \si_{n-1}(A) \ge \si_n(A) \ge 0.$$
We denote by $\GL$ the space of matrices $A \in \Knm$ with maximal rank : $\rank A = n$, that is $\si_n(A) > 0$ so that the set of singular matrices is 
$$\CN = \Knm \setminus \GL = \left\{ A \in \Knm \ : \ \si_n(A) = 0 \right\} .$$
Since the smallest singular value of a matrix is equal to the distance from the set of singular matrices:
$$\si_n(A) = d_F(A, \CN) = \min_{S \in \CN} \left\| A-S \right\|_F,$$
given an absolutly continuous path $A(t)$, $a \le t \le b$, we define its ``condition length'' by the integral
$$L_\ka = \int_a^b \left\|\frac{dA(t)}{dt}\right\|_F \si_n(A(t))^{-1}dt.$$
A good compromise between length and distance to $\CN$ is obtained in minimizing $L_\ka$. We call  ``minimizing condition geodesic'' an absolutly continuous path, parametrized by arc length,  which minimizes $L_\ka$ in the set of absolutly continuous paths with given end-points and condition distance $d_\ka(A,B)$ between two matrices the length $L_\ka$ of a minimizing condition geodesic with endpoints $A$ and $B$, if any. 

In this paper our objective is to investigate the properties of the smallest singular value $\sigma_n(A(t))$ along a condition geodesic. Our main result says that the map $\log \left(\sigma_n(A(t))^{-1}\right)$ is convex. Thus $\sigma_n(A(t))$ is concave, and its minimum value along the path is reached at one of the endpoints.

Note that a similar property holds in the case of hyperbolic geometry where instead of $\Knm$ we take $\R^{n-1} \times [0, \infty[$, instead of $\CN$ we have 
$\R^{n-1} \times \left\{ 0 \right\}$, and where the length of a path $a(t) = (a_1(t), \ldots , a_n(t))$ is defined by the integral 
$$\int \left\|\frac{da(t)}{dt}\right\| a_n(t)^{-1}dt.$$
Geodesics in that case are arcs of circles centered at $\R^{n-1} \times \left\{0\right\}$ or segments of vertical lines, and $\log \left( a_n(t)^{-1} \right)$ is convex along such paths. 

The approach used here to prove our theorems is heavily based on Riemannian geometry. We define on $\GL$ the following Riemannian structure:
$$\left\langle M,N \right\rangle_{\ka, A} = \si_n(A)^{-2} {\rm Re}\left\langle M,N \right\rangle_F$$
where $M, N \in \Knm$ and $A \in \GL$. The minimizing condition geodesics defined previously are clearly geodesic in $\GL$ for this Riemannian structure so that we may use the toolbox of Riemannian geometry. In fact things are not so simple: the smallest singular value $\sigma_n(A)$ is a locally Lipschitz map in $\GL$, and it is smooth on the open subset
$$\GLS = \left\{ A \in \GL \ : \ \sigma_{n-1}(A) > \sigma_{n}(A)\right\}$$
that is when the smallest singular value of $A$ is simple. On the open subset $\GLS$ the metric $\left\langle \cdot , \cdot \right\rangle_\ka$ defines a smooth Riemannian structure, and we call ``condition geodesics'' the geodesics related to this structure. Such a path is not necessarily a minimizing geodesic. Our first main theorem establishes a remarkable property of the condition Riemannian structure:

\begin{theorem} \label{th-1} $\si_n^{-2}$ is logarithmically convex on $\GLS$ i.e. for any geodesic curve $\ga(t)$ in $\GLS$ for the condition metric the map 
$\log \left(\si_n^{-2} (\ga(t))\right)$ is convex. 
\end{theorem}

\begin{problem}\label{prob-1}
The condition Riemannian structure $\left\langle .,. \right\rangle_\ka$ is defined in $\GL$ where it is is only locally Lipschitz. Let us define condition geodesics in $\GL$ as the extremals of the condition length $L_\ka$ (see for example 
\cite{cla} Chapter 4, Theorem 4.4.3, for the definition of such extremals in the Lipschitz case). 
Is Theorem \ref{th-1} still true for $\GL$? All the examples we have studied confirm that convexity holds, even if $\si_n^{-1} (\ga(t))$ fails to be $C^1$. See Boito-Dedieu \cite{boi}. We intend to address this issue in a future paper. 
\end{problem}

In a second step we extend these results to other spaces of matrices: the sphere $\S_r(\GLS )$ of radius $r$ in $\GLS$ in Corollary \ref{cor-3b}, the projective space $\P\left( \GLS \right)$ in Corollary \ref{cor-3c}. We also consider the case of the solution variety of the homogeneous equation 
$M \zeta = 0$ that is the set of pairs
$$\left\{(M, \zeta) \in {\mathbb K}^{n \times (n+1)} \times {\mathbb K}^{n + 1} \ : \ M \zeta = 0 \right\} .$$
Now our function $\al$ is the square of the condition number studied by Demmel in \cite{dem}.
This is done in the affine context in Theorem \ref{th-3} and in the projective context in Corollary \ref{cor-4}. 

\vskip 3mm Since $\si_n(A)$ is equal to the distance from $A$ to the set of singular matrices a natural question is to ask whether our main result remains valid for the inverse of the distance from certain sets or for more general functions. 

\begin{definition} \label{def-3} Let $(\CM,\langle\cdot,\cdot\rangle)$ be Riemannian and let $\al : \CM \rightarrow \R$ be a function of class $C^2$ with positive values. Let $\CM_\ka$ be the manifold $\CM$ with the new metric
$$\langle\cdot,\cdot\rangle_{\ka , x} = \al (x)\langle\cdot,\cdot\rangle_x $$ 
called condition Riemann structure. 
We say that $\al$ is {\it self-convex} when $\log \al(\ga (t))$ is convex for any geodesic $\ga$ in  $\CM_\ka$. 
\end{definition}

For example, with $\CM = \left\{ x=(x_1, \ldots , x_n) \in \R^n \ : \ x_n>0 \right\}$ equipped with the usual metric, $\al(x) = x_n^{-2}$ is self-convex. The space $\CM_\ka$ is the Poincar\'e model of hyperbolic space.

In the following theorem we prove self-convexity for the distance function to a $C^2$ submanifold without boundary $\CN \subset \R^j$. Let us denote by
$$\rho(x) = d(x,\CN) = \min_{y \in \CN} \left\| x-y \right\| \ \mbox{and} \ \al(x) = \frac{1}{\rho(x)^2}.$$
Let $\CU$ be the largest open set in $\R^j$ such that, for any $x \in \CU$, there is a unique closest point in $\CN$ to $x$. When $\CU$ is equipped with the new metric $\al(x)\left\langle .,. \right\rangle$ we have:

\begin{theorem} \label{th-2} The function $\al : \CU \setminus \CN \rightarrow \R$ is self-convex.
\end{theorem}

Theorem \ref{th-2} is then extended to the projective case. Let $\CN$ be a $\mathcal{C}^2$ submanifold without boundary of $\mathbb{P}(\mathbb{R}^j)$. 
Let us denote by $d_R$ the Riemannian distance in projective space (points in the projective space are lines throught the origin and the distance $d_R$ between two lines is the angle they make). Let us denote $d_\mathbb{P}=\sin d_R$ (this is also a distance), define $\al(x)=d_\mathbb{P}(x,\CN)^{-2}$, and let $\CU$ be the largest open subset of $\mathbb{P}(\mathbb{R}^j)$ such that for $x\in \CU$ there is a unique closest point from $\CN$ to $x$ for the distance $d_\P$. Then

\begin{corollary}\label{cor-th-2}
The map $\al : \CU \setminus \CN \rightarrow \mathbb{R}$ is self-convex.
\end{corollary}

The extension of Theorem \ref{th-1} and Theorem \ref{th-2} to other types of sets or functions is not obvious. In Example \ref{ex-Frob} we prove that
$\al(A)=\sigma_1(A)^{-2}+\cdots+\sigma_n(A)^{-2}$ is not self-convex in $\GL$.

In Example \ref{ex-circle} we take $\CN = \R^2$, and $\CU$ the unit disk so that $\CU$ contains a point (the center) which has many closest points from $\CN$. In that case the corresponding function $\al : \CU \setminus \CN \rightarrow \R$ is self-convex but it fails to be smooth at the center of the disk. 

In Example \ref{ex-two-points} we provide an example of a submanifold $\CN \subset \R^2$ such that the function $\al(x) = d(x, \CN)^{-2}$ defined on  $\R^2 \setminus \CN$ is not self-convex. 

Our interest in considering the condition metric in the space of matrices comes from recent papers by Shub \cite{shu} and Beltr\'an-Shub \cite{bel} where these authors use condition length along a path in certain solution varieties to estimate step size for continuation methods to follow these paths. They give bounds on the number of steps required in terms of the condition length of the path. If geodesics in the condition metric are followed the known bounds on polynomial system solving are vastly improved. To understand the properties of these geodesics we have begun in this paper with linear systems where we can investigate their properties more deeply. We find self-convexity in the context of this paper remarkable. We do not know if similar issues may naturally arise in linear algebra even for solving systems of linear equations. Similar issues do clearly arise when studying continuation methods for the eigenvalue problem.

\section{Self-convexity}

Let us first start to recall some basic definitions about convexity on Riemannian manifolds. A good reference on this subject is 
Udri\c ste \cite{udr}.

\begin{definition} \label{def-2} We say that a function $f : \CM \rightarrow \R$ is {\it convex} whenever 
$$f(\ga_{xy}(t)) \leq (1-t)f(x) + t f(y)$$
for every $x, y \in \CM$, for every geodesic arc $\ga_{xy}$ joigning $x$ and $y$ and $0 \le t \le 1$. 
\end{definition}

The convexity of $f$ in $\CM$ is equivalent to the convexity in the usual sense of $f \circ \ga_{xy}$ on $[0,1]$ for every $x,y \in U$ and the geodesic $\ga_{xy}$ joining $x$ and $y$ or also to the convexity of $g \circ \ga$  for every geodesic $\ga$ (\cite{udr} Chap. 3, Th. 2.2). Thus, we see that 

\begin{lemma} Self-convexity of a function $\al : \CM \ra \R$ is equivalent to the convexity of $\log \circ \al$ in the condition Riemannian manifold $\CM_\ka$. 
\end{lemma}

\vskip 3mm When $f$ is a function of class $C^2$ in the Riemannian manifold $\CM$, we define its second derivative $D^2 f(x)$ as the second covariant derivative. It is a symmetric bilinear form on $T_x \CM$. Note (\cite[Chapter 1]{udr}) that if $x\in\CM$ and $\dot x\in T_x\CM$, and if $\gamma(t)$ is a geodesic in $\CM$, $\gamma(0)=x,\frac{d}{dt}\gamma(0)=\dot x$, then
\[
D^2f(x)(\dot x,\dot x)=\frac{d^2}{dt^2}(f\circ\gamma)(0).
\]
This second derivative depends on the Riemannian connection on $\CM$. Since $\CM$ is equipped with two different metrics: $\left\langle .,. \right\rangle$ and $\left\langle .,. \right\rangle_\ka$ we have to distinguish between the corresponding second derivatives; they are denoted by $D^2f(x)$ and $D_\ka^2f(x)$ respectively. No such distinction is necessary for the first derivative $Df(x)$.

Convexity on Riemannian manifold is characterized by (see \cite{udr} Chap. 3, Th. 6.2):

\begin{proposition}\label{pro-1} 
A function $f : \CM \rightarrow \R$ of class $C^2$ is convex if and only if $D^2f(x)$ is positive semidefinite for every $x \in \CM$.
\end{proposition}

We use this proposition to obtain a caracterisation of self-convexity: $\al$ is self-convex if and only if the second derivative $D_\ka^2(\log \circ \al)(x)$ is positive semidefinite for any $x \in \CM_\ka$. We get

\begin{proposition}\label{pro-2}
For a function $\al : \CM \rightarrow \R$ of class $C^2$ with positive values self-convexity is equivalent to
\[
2\al(x)D^2\al(x)(\x,\x) + \|D\al(x)\|_x^2  \|\x\|_x^2 - 4(D\al(x)\x)^2 \geq 0
\]
for any $x\in \CM$ and for any vector $\x \in T_x\CM$, the tangent space at $x$.
\end{proposition}

\begin{proof}
Let $x \in \CM$ be given. Let $\varphi : \R^m \rightarrow \CM$ be a coordinate system such that $\varphi(0)=x$ and with first fundamental form $g_{ij}(0)=\delta_{ij}$ (Kronecker's delta) and Christoffel's symbols $\Gamma_{jk}^i(0) = 0$, and let
$$A = \al \circ \varphi $$
so that $\al(x) = A(0)$. Those coordinates are called ``normal'' or ``geodesic''. Note that this implies
\[
\frac{\partial g_{ij}}{\partial z_k}(0)=0
\]
for all $i,j,k$. We denote by $g_{\ka,ij}$ and $\Gamma_{\ka,jk}^i$ respectively the first fundamental form and the Christoffel symbols for $\varphi$ in $\CM_\ka$. Let us compute them. Note that
\[
g_{\ka,ij}(z)=g_{ij}(z)A(z),
\]
\[
\frac{\partial g_{\ka,ij}}{\partial z_k}(0) = Dg_{\ka,ij}(0)(e_k) = D(g_{ij}A)(0)(e_k) =
\]
\[
g_{ij}(0)DA(0)(e_k) + A(0)Dg_{ij}(0)(e_k) = \delta_{ij}\frac{\partial A}{\partial z_k}(0).
\]
Moreover,
\[
\Gamma_{\ka,jk}^i = \frac{1}{A(0)} \Gamma_{jk}^i = \frac{1}{2A(0)}\left(\frac{\partial g_{\ka,ij}}{\partial z_k}(0) + \frac{\partial g_{\ka,ik}}{\partial z_j}(0) - 
\frac{\partial g_{\ka,jk}}{\partial z_i}(0)\right) =
\]
\[
\frac{1}{2A(0)}\left(\delta_{ij}\frac{\partial A}{\partial z_k}(0) + \delta_{ik}\frac{\partial A}{\partial z_j}(0) - 
\delta_{jk}\frac{\partial A}{\partial z_i}(0)\right).
\]
That is,
\[
\begin{cases}
\Gamma_{\ka, ik}^i = \Gamma_{\ka, ki}^i = \frac{1}{2A(0)}\frac{\partial A}{\partial z_k}(0)&\mbox{for all } i,k,\\
\Gamma_{\ka, jj}^i = \frac{-1}{2A(0)}\frac{\partial A}{\partial z_i}(0),&j\neq i,\\
\Gamma_{\ka, jk}^i = 0&\mbox{otherwise}.
\end{cases}
\]
The second derivative of the composition of two maps 
$$\CM \stackrel{f}{\rightarrow} \R \stackrel{\psi}{\rightarrow} \R$$
is given by the identity (see \cite{udr} Chap. 1.3, Hessian)
$$D^2(\psi \circ f)(x) = D\psi(f(x))D^2f(x) + \psi''(f(x))Df(x)\otimes Df(x)$$
and where $Df(x)\otimes Df(x)$ is the bilinear form on $T_x \CM $ by
$$(Df(x)\otimes Df(x))(u,v) = Df(x)(u) Df(x)(v).$$
This gives in our context, that is when $f = \al$ and $\psi = \log$, 
$$D_\ka^2(\log \circ \al)(x) = \frac{1}{\al(x)}D_\ka^2 \al(x) - \frac{1}{\al(x)^2}D\al(x)\otimes D\al(x).$$

\vskip 3mm According to Proposition \ref{pro-1} our objective is now to give a necessary and sufficient condition for $D_\ka^2(\log \circ \al)(x)$ to be positive semidefinite for each $x \in \CM$. 
In our system of local coordinates the components of $D^2\al(x)$ are (see \cite{udr} Chap. 1.3)
$$A_{jk} = \frac{\partial^2A}{\partial z_j \partial z_k} - \sum_i \Gamma_{jk}^i \frac{\partial A}{\partial z_i} = \frac{\partial^2A}{\partial z_j \partial z_k}$$
while the components of $D_\ka^2\al(x)$ are
$$A_{\ka,jk} = \frac{\partial^2A}{\partial z_j \partial z_k} - \sum_i \Gamma_{\ka,jk}^i \frac{\partial A}{\partial z_i}.$$
If we replace the Christoffel symbols in this last sum by the values previously computed we obtain, when $j=k$,
$$\sum_i \Gamma_{\ka,jj}^i \frac{\partial A}{\partial z_i} = \Gamma_{\ka,jj}^j \frac{\partial A}{\partial z_j} + 
\sum_{i \ne j} \Gamma_{\ka,jj}^i \frac{\partial A}{\partial z_i} = $$
$$\frac{1}{2A}\left( \frac{\partial A}{\partial z_j} \right)^2 - \frac{1}{2A}\sum_{i \ne j}\left( \frac{\partial A}{\partial z_i} \right)^2 = \frac{1}{A}\left( \frac{\partial A}{\partial z_j} \right)^2 - \frac{1}{2A}\sum_{i}\left( \frac{\partial A}{\partial z_i} \right)^2$$
while when $j \ne k$, 
$$\sum_i \Gamma_{\ka,jk}^i \frac{\partial A}{\partial z_i} = \Gamma_{\ka,jk}^j \frac{\partial A}{\partial z_j} + 
\Gamma_{\ka,jk}^k \frac{\partial A}{\partial z_k} = $$
$$\frac{1}{2A} \frac{\partial A}{\partial z_k} \frac{\partial A}{\partial z_j} + \frac{1}{2A} \frac{\partial A}{\partial z_j} \frac{\partial A}{\partial z_k} = \frac{1}{A} \frac{\partial A}{\partial z_j} \frac{\partial A}{\partial z_k}.$$
Both cases are subsumed in the identity
$$\sum_i \Gamma_{\ka,jk}^i \frac{\partial A}{\partial z_i} = \frac{1}{A} \frac{\partial A}{\partial z_j} \frac{\partial A}{\partial z_k} - \frac{\de_{jk}}{2A}\sum_{i}\left( \frac{\partial A}{\partial z_i} \right)^2.$$ 
Putting together all these identities gives the following expression for the components of $D_\ka^2(\log \circ \al)(x)$:
$$D_ka^2(\log \circ \al)(x)_{jk} = \frac{1}{A} \left(\frac{\partial^2A}{\partial z_j \partial z_k} -  \frac{1}{A} \frac{\partial A}{\partial z_j} \frac{\partial A}{\partial z_k} + \frac{\de_{jk}}{2A} \sum_i \left( \frac{\partial A}{\partial z_i} \right)^2\right)
- \frac{1}{A^2}\frac{\partial A}{\partial z_j} \frac{\partial A}{\partial z_k} =$$
$$\frac{1}{2A^2}\left(2A\frac{\partial^2A}{\partial z_j \partial z_k} + \de_{jk} \sum_i \left( \frac{\partial A}{\partial z_i} \right)^2 - 4\frac{\partial A}{\partial z_j} \frac{\partial A}{\partial z_k}\right)
.$$
Thus, $D_\ka^2(\log \circ \al)(x) \ge 0$ if and only if 
$$2\al(x)D^2\al(x) + \left\| D\al(x) \right\|_x^2\left\langle .,. \right\rangle_x - 4D\al(x)\otimes D\al(x)$$ 
is positive semi-definite, that is when
\[
2\al(x)D^2\al(x)(\x,\x) + \|D\al(x)\|_x^2 \|\x\|_x^2 - 4(D\al(x)\x)^2 \geq 0
\]
for any $x\in \CM$ and for any vector $\x \in T_x\CM$. This finishes the proof.
\end{proof}

\vskip 3mm An easy consequence of Proposition \ref{pro-2} is the following. See also Example \ref{ex-two-points}.

\begin{corollary}\label{cor-0}
When a function $\al : \CM \rightarrow \R$ of class $C^2$ is self-convex then any critical point of $\al$ has a positive semi-definite second derivative $D^2\al(x)$. Such a function cannot have a strict local maximum or a non-degenerate saddle.
\end{corollary}

\begin{proposition}\label{pro-3}
The following condition is equivalent for a $C^2$ function $\al=1/\rho^2 : \CM \longrightarrow \R$ to be self-convex on $\CM$: For every $x \in \CM$ and $\x\in T_x \CM$,
\[
\|\x\|^2\|D\rho(x)\|^2-(D\rho(x)\x)^2-\rho(x)D^2\rho(x)(\x,\x)\geq0,
\]
or, what is the same,
\[
2\|\x\|^2\|D\rho(x)\|^2\geq D^2\rho^2(x)(\x,\x).
\]
\end{proposition}

\begin{proof}
Note that
\[
D\al(x)\x=\frac{-2}{\rho(x)^3}D\rho(x)\x,
\]
\[
D^2\al(x)(\x,\x)=\frac{6}{\rho(x)^4}(D\rho(x)\x)^2-\frac{2}{\rho(x)^3}D^2\rho(x)(\x,\x).
\]
Hence, the necessary and sufficient condition of Proposition \ref{pro-2} reads
\[
\frac{4\|\x\|^2\|D\rho(x)\|^2}{\rho(x)^6}-\frac{16}{\rho(x)^6}(D\rho(x)\x)^2+\frac{12}{\rho(x)^6}(D\rho(x)\x)^2-\frac{4}{\rho(x)^5}D^2\rho(x)(\x,\x)\geq0,
\]
and the proposition follows.
\end{proof}

\begin{corollary}\label{cor-1}
Each of the following conditions is sufficient for a function $\al = 1/\rho^2:\CM\longrightarrow\R$ to be self-convex at $x\in \CM$: For every $\x\in T_x\CM$,
\[
D^2\rho(x)(\x,\x)\leq0,
\]
or
\[
\|D^2\rho^2(x)\|\leq2\|D\rho(x)\|^2.
\]
\end{corollary}

In the following proposition we obtain a weaker condition on $\al$ to obtain convexity in $\CM_\ka$ instead of self-convexity.

\begin{proposition}\label{prop:convexnolog}
$\al(x)$ is convex in $\CM_\ka$ if and only if
\[
2\al(x)D^2\al(x)(\x,\x) + \|D\al(x)\|_x^2\|\x\|_x^2 - 2(D\al(x)\x)^2\geq0,
\]
for any $x\in\CM$ and any vector $\x \in T_x\CM$.
\end{proposition}
\begin{proof}
We follow the lines of the proof of Proposition \ref{pro-2} with $\psi$ equal to the identity map instead of $\psi = \log$. 
\end{proof}

\section{Some general formulas for matrices}

\begin{proposition}\label{pro-4}
Let $A = (\Si, 0) \in \GLS$, where $\Si = \diag (\si_1\geq\cdots\geq \si_{n-1}>\si_n) \in \K^{n \times n}$. The map $\si_n : \GLS \rightarrow \R$ is a smooth map and, for every $U \in \K^{n \times m}$, 
$$\left\{
\begin{array}[pos]{l}
	D\si_n(A)U = \mathrm{Re}(u_{nn}),\\
	D^2\sigma^2_n(A)(U,U) = 2\sum_{j=1}^m |u_{nj}|^2 - 2\sum_{k=1}^{n-1}\frac{|u_{kn}\si_n + \overline{u_{nk}}\si_k|^2}{\si_k^2-\si_n^2}.
\end{array}
\right.
$$
\end{proposition}

\begin{proof} Since $\si_n^2$ is an eigenvalue of $AA^*$ with multiplicity $1$, the implicit function theorem proves the existence of smooth functions $\si^2_n(B) \in \R$ and $u(B) \in \K^n$, defined in an open neighborhood of $A$ and satisfying
$$\left\{
\begin{array}[pos]{l}
	BB^*u(B) = \si^2_n(B)u(B),\\
	\left\| u(B) \right\|^2 = 1,\\
	u(A) = e_n = (0, \ldots ,0,1)^T \in \K^n,\\
	\si^2_n(A) = \si_n^2.
\end{array}
\right.
$$
Differentiating these equations at $B$ gives, for any $U \in \K^{n \times m}$, 
$$\left\{
\begin{array}[pos]{l}
	(UB^*+BU^*)u(B) + BB^*\u(B)= \left( {\sigma_n^2} \right)' u(B) + \sic(B) \u (B),\\
	u(B)^*\u(B) = 0
\end{array}
\right.
$$
with $\u (B) = Du(B)U$ and $ \left( {\sigma_n^2} \right)' = D\sic (B)U$. Pre-multiplying the first equation by $u(B)^*$ gives
$$u(B)^*(UB^*+BU^*)u(B) + u(B)^*BB^*\u(B)=  \left( {\sigma_n^2} \right)' u(B)^*u(B) + \sic(B) u(B)^* \u (B)$$
so that
$$D\sic (B)U =  \left( {\sigma_n^2} \right)' = 2 \mathrm{Re} (u(B)^* UB^* u(B))$$
and
$$D \si_n(B)U = \frac{\mathrm{Re} (u(B)^* UB^* u(B))}{\si_n(B)}.$$
The derivative of the eigenvector is now easy to compute:
$$D u(B)U = \u (B) = (\sic (B) I_n - BB^*)^\dagger (UB^*+BU^* -  \left( {\sigma_n^2} \right)' I_n)u(B)$$
where $(\sic (B) I_n - BB^*)^\dagger$ denotes the generalized inverse (or Moore-Penrose inverse) of $\sic (B) I_n - BB^*$. 

The second derivative of $\sic$ at $B$ is given by
$$D^2\sic(B)(U,U) = 2 \mathrm{Re} (\u(B)^*UB^* u(B) + u(B)^* UU^* u(B) + u(B)^*UB^* \u(B)) = $$
$$2 \mathrm{Re} (u(B)^* UU^* u(B) + u(B)^*(UB^*+BU^*) \u(B)) = 2 \mathrm{Re} (u(B)^* UU^* u(B) + $$
$$u(B)^*(UB^*+BU^*)(\sic (B) I_n - BB^*)^\dagger (UB^*+BU^* -  \left( {\sigma_n^2} \right)' I_n)u(B)).$$
Using $u(A) = e_n$ and $\si_n(A) = \si_n$ we get 
$$\left\{
\begin{array}[pos]{l}
	D\sic (A)U = 2 \mathrm{Re} (UA^*)_{nn} = 2 \si_n \mathrm{Re}( u_{nn}),\\
	D\si_n (A)U = \mathrm{Re}( u_{nn}),\\
\end{array}
\right.
$$
and the second derivative is given by
$$D^2\sic(A)(U,U) = $$
$$2 \mathrm{Re} \left( (UU^*)_{nn} + \sum_{k=1}^{n-1}(UA^*+AU^*)_{nk}(\si_n^2 - \si_k^2)^{-1}(UA^*+AU^* -  \left( {\sigma_n^2} \right)' I_n)_{kn} \right) = $$
$$2 \mathrm{Re} \left( (UU^*)_{nn} + \sum_{k=1}^{n-1}\frac{\left|(UA^*+AU^*)_{kn}\right|^2}{\si_n^2 - \si_k^2} \right) = 
2 \sum_{j=1}^{m} \left| u_{nj} \right|^2 - 2 \sum_{k=1}^{n-1}\frac{\left| u_{kn}\si_n + \overline{u_{nk}}\si_k \right|^2}{\si_k^2 - \si_n^2}.$$
\end{proof}

\begin{corollary}\label{cor-2}
Let $A = (\Si, 0) \in \GLS$, where $\Si = \diag(\si_1\geq\cdots\geq \si_{n-1}>\si_n>0) \in \K^{n \times n}$. Let us define $\rho(A) = \si_n(A) / \left\| A \right\|_F$. Then, for any $U \in \K^{n \times m}$ such that $\mathrm{Re} \left\langle A,U \right\rangle_F = 0$, we have 
$$\left\{
\begin{array}[pos]{l}
	D\rho (A)U = \frac{\mathrm{Re}( u_{nn})}{\left\| A \right\|_F},\\
	D^2\rho^2 (A)(U,U) = \frac{2}{\left\| A \right\|_F^2}\left(\sum_{j=1}^m |u_{nj}|^2 - \sum_{k=1}^{n-1}\frac{|u_{kn}\si_n + \overline{u_{nk}}\si_k|^2}{\si_k^2-\si_n^2} - \frac{\left\| U \right\|_F^2}{\left\| A \right\|_F^2}\si_n^2\right).\\
\end{array}
\right.
$$
\end{corollary}
\begin{proof}
Note that
\[
D\rho(A)U = \frac{D\sigma_n(A)U \|A\|_F - \sigma_n(A)\frac{2\mathrm{Re}\langle A,U\rangle_F}{2\| A \|_F}}{\| A \|_F^2}=\frac{D\sigma_n(A)U}{\| A \|_F},
\]
and the first assertion of the corollary follows from Proposition \ref{pro-4}.
For the second one, note that $h={h_1}/{h_2}$ (for real valued $\mathcal{C}^2$ functions $h,h_1,h_2$ with $h_2(0)\neq0$) implies
\[
D^2h = \frac{h_2^2D^2h_1 - h_1h_2D^2h_2 - 2h_2Dh_1Dh_2 + 2h_1(Dh_2)^2}{h_2^3}.
\]
Now, $\rho^2(A)={\sigma_n^2(A)}/{\| A \|_F^2}$, $D(\| A \|_F^2)U = 2\mathrm{Re}\langle A,U \rangle_F = 0,$ $D^2(\| A \|_F^2)(U,U) = 2\| U \|_F^2,$
and $D^2\sigma_n^2(A)(U,U)$ is known from Proposition \ref{pro-4}. The formula for $D^2\rho^2 (A)$ follows after some elementary calculations.
\end{proof}

\section{The affine linear case}

We consider here the Riemannian manifold $\CM = \GLS$ equipped with the usual Frobenius Hermitian product.
Let $\al :\GLS \rightarrow \R$ be defined as $\al(A) = {1}/{\sigma_n^2(A)}$. 

\begin{corollary}\label{cor-3} The function $\al$ is self-convex in $\GLS$.
\end{corollary}

\begin{proof}
From Proposition \ref{pro-3}, it suffices to see that
\[
2\| U \|_F^2 \| D\sigma_n(A) \|_F^2  \geq D^2\sigma_n^2(A)(U,U).
\]
Since unitary transformations are isometries in $\GLS$ with respect to the condition metric we may suppose, via a singular value decomposition that $A = (\Si, 0) \in \GLS$, where $\Si = \diag(\si_1 \geq \cdots \geq \si_{n-1} > \si_n) \in \K^{n \times n}$. Now, the inequality to verify is obvious from Proposition \ref{pro-4}, as $\| D\sigma_n(A) \|_F = 1$ and
\[
D^2\sigma_n^2(A)(U,U) = 2 \sum_{j=1}^m|u_{nj}|^2 - 2\sum_{k=1}^{n-1}\frac{| u_{kn}\si_n + \overline{u_{nk}}\si_k |^2}{\si_k^2 - \si_n^2} \leq 2\sum_{j=1}^m |u_{nj}|^2 \leq 2 \| U \|_F^2.
\]
\end{proof}

\begin{corollary}\label{cor-3b} Let $r>0$. The function $\al$ is
self-convex in the sphere $\S_r(\GLS)$ of radius $r$ in $\GLS$.
\end{corollary}

\begin{proof}
It is enough to prove that any geodesic in $(\S_r(\GLS), \al)$ is also
a geodesic in $(\GLS, \al)$. Indeed, suppose that $A$ and $B$ are
matrices in $\S_r(\GLS)$ and the minimal geodesic in $(\GLS, \al)$
between $A$ and $B$ is $X(t)$, $a \le t \le b$. Then we claim that
$L_\ka\left(\frac{rX(t)}{\left\| X(t) \right\|_F}\right) \le
L_\ka(X(t))$. Indeed, for any $t$,
$$\frac{d}{dt}\left(\frac{rX(t)}{\left\| X(t) \right\|_F}\right) =
\frac{r\frac{dX(t)}{dt}}{\left\| X(t) \right\|_F} -
r\frac{X(t)Re(\langle X(t),\frac{dX(t)}{dt}\rangle_F)}{\left\| X(t)
\right\|_F^3}$$
so that
$$
\left\|\frac{d}{dt}\left(\frac{rX(t)}{\|X(t)\|}\right)\right\|_F=
$$
$$
\left(\frac{r^2\left\|\frac{dX(t)}{dt} \right\|_F^2}{\left\| X(t)
\right\|_F^2}+\frac{r^2Re(\langle
X(t),\frac{dX(t)}{dt}\rangle_F)^2}{\left\| X(t)
\right\|_F^4}-\frac{2r^2Re(\langle
X(t),\frac{dX(t)}{dt}\rangle_F)^2}{\left\| X(t)
\right\|_F^4}\right)^{1/2}=
$$
$$
\left(\frac{r^2\left\|\frac{dX(t)}{dt} \right\|_F^2}{\left\| X(t)
\right\|_F^2}-\frac{r^2Re(\langle
X(t),\frac{dX(t)}{dt}\rangle_F)^2}{\left\| X(t)
\right\|_F^4}\right)^{1/2}\leq \frac{r\left\|\frac{dX(t)}{dt}
\right\|_F}{\left\| X(t) \right\|_F}
$$
Hence,
$$\left\| \frac{d}{dt}\left(\frac{rX(t)}{\left\| X(t)
\right\|_F}\right) \right\|_\ka =
\si_n^{-1}\left(\frac{rX(t)}{\left\| X(t) \right\|_F}\right)
\left\|\frac{d}{dt}\left(\frac{rX(t)}{\|X(t)\|}\right)\right\|_F=
$$
$$\frac{\|X(t)\|_F\si_n^{-1}\left(X(t)\right)}{r}
\left\|\frac{d}{dt}\left(\frac{rX(t)}{\|X(t)\|}\right)\right\|_F\le
\si_n^{-1}\left(X(t)\right)\left\| \frac{dX(t)}{dt} \right\|_F=
\left\| \frac{dX(t)}{dt} \right\|_\ka .$$
Therefore $X(t)$ can only be a minimizing geodesic if it belongs to
$\S_r(\GLS)$. Since all geodesics are locally minimizing geodesics,
Corollary \ref{cor-3b} follows.
\end{proof}

\vskip 3mm The following gives an example of a smooth and non-selfconvex function in $\GL$. 

\begin{example}\label{ex-noconvexmuf}\label{ex-Frob}
For $n \ge 3$, the function $\al(A)=\sigma_1(A)^{-2}+\cdots+\sigma_n(A)^{-2}$ is not self-convex in $\GL$.
\end{example}
\begin{proof}
For simplicity we consider the case of real square matrices. We have $\al(A)=\|A^{-1}\|_F^2$,
\[
D\al(A)\dot{A}=-2\langle A^{-1},A^{-1}\dot{A}A^{-1}\rangle_F=-2\langle A^{-T}A^{-1}A^{-T},\dot{A}\rangle_F,
\]
\[
\|D\al(A)\|_F^2=4\|A^{-T}A^{-1}A^{-T}\|^2_F,
\]
\[
D^2\al(A)(\dot{A},\dot{A})=2\|A^{-1}\dot{A}A^{-1}\|^2_F+4\langle A^{-1},A^{-1}\dot{A}A^{-1}\dot{A}A^{-1}\rangle_F.
\]
According to Proposition \ref{prop:convexnolog}, the self-convexity of $\al(A)$ in $\mathbb{GL}_n$ is equivalent to 
\[
2\|A^{-1}\|_F^2\left(2\|A^{-1}\dot{A}A^{-1}\|^2_F+4\langle A^{-1},A^{-1}\dot{A}A^{-1}\dot{A}A^{-1}\rangle_F\right)+
\]
\[
4\|\dot{A}\|_F^2\|A^{-T}A^{-1}A^{-T}\|^2_F-8\langle A^{-1},A^{-1}\dot{A}A^{-1}\rangle_F^2\geq0
\]
This inequality is not satisfied when
\[
A=\begin{pmatrix}1&0&0\\0&1&0\\0&0&2\end{pmatrix}\text{   and   }\dot{A}=\begin{pmatrix}0&1&0\\-1&0&0\\0&0&0\end{pmatrix}.
\]
\end{proof}

\section{The homogeneous linear case}

\subsection{The complex projective space.} 




The matter of this subsection is mainly taken from Gallot-Hulin-Lafontaine \cite{gal} sect. 2.A.5. 

Let $V$ be a Hermitian space of complex dimension $\dim_\C V = d+1$. We denote by $\P(V)$ the corresponding projective space that is the quotient of $V \setminus \left\{ 0 \right\}$ by the group $\C^*$ of dilations of $V$; $\P(V)$ is equipped with its usual smooth manifold structure with complex dimension $\dim \P(V) = d$. We denote by $p$ the canonical surjection.

Let $V$ be considered as a real vector space of dimension $\dim_\R V = 2d+2$ equipped with the scalar product $\mathrm{Re}\left\langle .,. \right\rangle_V$.
The sphere $\S(V)$ is a submanifold in $V$ of real dimension $2d+1$. This sphere being equipped with the induced metric becomes a Riemannian manifold and, as usual, we identify the tangent space at $z \in \S(V)$ with
$$T_z \S(V) = \left\{ u \in V \ : \ \mathrm{Re}\left\langle u,z \right\rangle_V = 0 \right\}.$$

The projective space $\P(V)$ can also be seen as the quotient $\S(V) / S^1$ of the unit sphere in $V$ by the unit circle in $\C$ for the action given by $(\la , z) \in S^1 \times \S(V) \rightarrow \la z \in \S(V)$. The canonical map is denoted by
$$p_V : \S(V) \rightarrow \P(V).$$
$p_V$ is the restriction of $p$ to $\S(V)$.

The horizontal space at $z \in \S(V)$ related to $p_V$ is defined as the (real) orthogonal complement of $\ker Dp_V(z)$ in $T_z \S(V)$. This horizontal space is denoted by $H_z$. Since $V$ is decomposed in the (real) orthogonal sum 
$$ V = \R z \oplus \R iz \oplus z^\perp $$
and since $\ker Dp_V(z) = \R iz$ (the tangent space at $z$ to the circle $S^1 z$) we get 
$$H_z  = z^\perp = \left\{ u \in V \ : \ \left\langle u,z \right\rangle = 0 \right\}.$$

There exists on $\P(V)$ a unique Riemannian metric such that $p_V$ is a Riemannian submersion that is, $p_V$ is a smooth 
submersion and, for any $z \in \S(V)$, $Dp_V(z)$ is an isometry between $H_z$ and $T_{p(z)}\P(V)$. Thus, for this Riemannian structure, one has:
$$\left\langle Dp_V(z)u, Dp_V(z)v \right\rangle_{T_{p(z)}\P(V)} = \mathrm{Re} \left\langle u,v \right\rangle_V$$
for any $z \in \S(V)$ and $u,v \in H_z$. 

\begin{proposition} \label{pro-CP} Let $z \in \S(V)$ be given. 
\begin{enumerate}
\item A chart at $p(z) \in \P(V)$ is defined by 
$$\varphi_z : H_z \rightarrow \P(V), \ \ \varphi_z(u) = p(z+u).$$

\item Its derivative at $0$ is the restriction of $Dp(z)$ at $H_z$: 
$$D\varphi_z (0) = Dp(z) : H_z \rightarrow T_{p(z)}\P(V)$$
which is an isometry. 

\item For any smooth mapping $\psi : \P(V) \rightarrow \R$, and for any  $v \in H_z$ we have
$$D\psi(p(z))\left(Dp(z) v\right) = D(\psi \circ \varphi_z)(0)v$$
and
$$D^2\psi(p(z))(Dp(z)v, Dp(z)v) = D^2(\psi \circ \varphi_z)(0)(v,v).$$
\end{enumerate}
\end{proposition}

\begin{proof} 1 and 2 are easy. We have $D(\psi \circ \varphi_z)(0) = D\psi(p(z))D(\varphi_z)(0)$ which gives 3 since 
$D(\varphi_z)(0)v = Dp(z)v$ for any $v \in H_z$. For the second derivative, recall that $D^2\psi(p(z))(Dp(z)v,Dp(z)v)=(\psi\circ\tilde{\gamma})''(0)$, where $\tilde{\gamma}$ is a geodesic curve in $\P(V)$ such that $\tilde{\gamma}(0)=p(z),\tilde{\gamma}'(0)=Dp(z)v$. Now, consider the horizontal $p_V-$lift $\gamma$ of $\tilde{\gamma}$ to $\S(V)$ with base point $z$. Note that $\gamma(0)=z,\gamma'(0)=v$. Hence,
\[
(\psi\circ\tilde{\gamma})''(0)=(\psi\circ p\circ \gamma)''(0)=
D^2(\psi\circ p)(z)(v,v)+D\psi(p(z))Dp(z)\gamma''(0).
\]
As $\gamma''(0)$ is orthogonal to $T_z \S(V)$, we have $Dp(z)\gamma''(0)=0$. Finally,
\[
D^2(\psi\circ p)(z)(v,v)=(\psi\circ p (z+tv))''(0)=(\psi\circ\varphi_z(tv))''(0)=D^2(\psi\circ\varphi_z)(0)(v,v),
\]
and the assertion on the second derivative follows.
\end{proof}

The following result will be helpful.

\begin{proposition}\label{pro-RS}
Let $\CM_1,{\CM}_2$ be Riemannian manifolds and $\al_2:\CM_2 \rightarrow ]0,\infty[$ be of class ${C}^2$. Let $\pi:\CM_1 \rightarrow \CM_2$ be a Riemannian submersion. Let $\CU_2 \subseteq \CM_2$ be an open set and assume that $\al_1 = \al_2 \circ \pi$ is self-convex in $\CU_1 = \pi^{-1}(\CU_2)$. Then, $\al_2$ is self-convex in $\CU_2$.
\end{proposition}

\begin{proof}
Let $\CM_{\ka,1}$ be $\CM_1$, but endowed with the condition metric given by $\al_1$, and let $\CM_{\ka,2}$ be $\CM_2$, but endowed with the condition metric given by $\al_2$. Then, $\pi : \CM_{\ka,1} \rightarrow \CM_{\ka,2}$ is also a Riemannian submersion.

Now, let $\gamma_2 : [a,b] \rightarrow \CU_2 \subseteq \CM_{\ka,2}$ be a geodesic, and let $\gamma_1 \subseteq \CM_{\ka,1}$ be its horizontal lift by $\pi$. Then, $\gamma_1$ is a geodesic in $\CU_1 \subseteq \CM_1$ (see \cite[Cor 2.109]{gal}) and hence $\log \al_1(\ga_1(t))$ is a convex function of $t$. Now,
\[
\log(\al_2(\gamma_2(t))) = \log(\al_2 \circ \pi(\ga_1 (t))) = \log(\al_1(\ga(t))),
\]
is convex as wanted. \end{proof}

\begin{corollary}\label{cor-3c} The function $\al_2 : \P(\GLS) \rightarrow \R, \; \al_2(A) = \|A\|_F^2 \si_n^{-2}(A)$ is self-convex in $\P(\GLS)$.
\end{corollary}

\begin{proof}
Note that $p : \S(\GLS) \rightarrow \P(\GLS)$ is a Riemannian submersion and $\al_2 = \al \circ p$ where $\al$ is as in Corollary \ref{cor-3b}. The corollary follows from Proposition \ref{pro-RS}.
\end{proof}

\subsection{The solution variety.} 
Let us denote by $p_1$ and $p_2$ the canonical maps 
$$\S_1 \stackrel{p_1}{\rightarrow} \P\left(\K^{n \times (n+1)}\right) \ \mbox{ and } \ 
\S_2 \stackrel{p_2}{\rightarrow} \P\left(\K^{n+1}\right) = \P_n (\K), $$
where $\S_1$ is the unit sphere in $\K^{n \times (n+1)}$ and $\S_2$ is the unit sphere in $\K^{n+1}$. Consider the affine solution variety,
$$\hat{\CW}^> = \left\{(M,\ze) \in  \S_1 \times \S_2 \ : \ M \in \mathbb{GL}_{n,n+1}^> \mbox{ and } M\ze = 0\right\}.$$
It is a Riemannian manifold equipped with the metric induced by the product metric on $\K^{n\times(n+1)}  \times \K^{n+1}$. The tangent space to $\hat\CW^>$ is given by
$$T_{(M,\ze)} \hat{\CW}^> = \left\{(\dot M, \dot \ze) \in T_M\S_1 \times T_\ze \S_2 \ : \ \dot M \ze + M \dot \ze = 0 \right\}.$$
The projective solution variety considered here is
$$\CW^> = \left\{(p_1(M),p_2(\ze)) \in  \PKn \times \Pn \ : \ M \in \mathbb{GL}_{n,n+1}^>  \mbox{ and } M\ze = 0\right\},$$
that is also a Riemannian manifold equipped with the metric induced by the product metric on $\PKn \times \Pn$.

\vskip 3mm Let us denote by $\pi_1$ the restriction to $\hat{\CW}^>$ of the first projection $\S_1 \times \S_2 \rightarrow \S_1$,
and by $R : \hat{\CW}^> \rightarrow \R$, $R = \sigma_n \circ \pi_1.$ We have

\begin{lemma}\label{lem-der1dies}
Let $w = ( M, \zeta) \in \hat\CW^>$ and let $\gamma$ be a geodesic in $\hat\CW^>$, $\gamma(0)=w$. Then,
\[
D \sigma_n(\pi_1(w))(\pi_1\circ\gamma)''(0)<0.
\]
\end{lemma}
\begin{proof} Our problem is invariant by unitary change of coordinates. Hence, using a singular value decomposition,  we can assume that $M = (\Si, 0) \in \mathbb{GL}_{n,n+1}^>$, where $\Si = \diag(\si_1\geq\cdots\geq \si_{n-1}>\si_n) \in \K^{n \times n}$ and $\ze = e_{n+1} = (0, \ldots , 0, 1)^T \in \S_2$. As $\gamma=(M(t),\zeta(t))$ is a geodesic of $\hat\CW^>\subseteq\K^{n\times(n+1)}\times\K^n$, $\gamma''(0)$ is orthogonal to $T_w\hat{\CW}$, which contains all the pairs of the form $((A,0),0)$ where $A$ is a $n\times n$ matrix, ${\rm Re}\langle \Sigma , A \rangle=0$. Hence, $M''(0)$ has the form
\[
M''(0)=(a \Si,*),
\]
for some real number $a\in\R$. Finally, $M(t)$ is contained in the sphere so $\left\| M(t) \right\|_F = 1$ and
\[
0=(||M(t)||_F^2)''(0)=2||M'(0)||_F^2+2{\rm Re}\langle M(0),M''(0)\rangle=2||M'(0)||_F^2+2a,
\]
so that $a=-\|M'(0)\|_F^2$ and $(M''(0))_{nn}=-\|M'(0)\|_F^2\sigma_n$. From Proposition \ref{pro-4},
\[
D \sigma_n(\pi_1(w))(\pi_1\circ\gamma)''(0)={\rm Re}((\pi_1\circ\gamma)''(0)_{nn})={\rm Re}(M''(0))_{nn}<0.
\]
\end{proof}

\begin{theorem} \label{th-3} The map $\al : \hat\CW^> \rightarrow \R$ given by $\al(M,\zeta) = \sigma_n(M)^{-2}$ is self-convex. 
\end{theorem}

\begin{proof} Using unitary invariance we can take $M = (\Si, 0) \in \mathbb{GL}_{n,n+1}^>$, where $\Si = \diag(\si_1\geq\cdots\geq \si_{n-1}>\si_n) \in \K^{n \times n}$ and $\ze = e_{n+1} = (0, \ldots , 0, 1)^T \in \S_2$. According to proposition \ref{pro-3} we have to prove that 
$$2 \left\| \dot w \right\|_w^2 \left\| DR(w) \right\|^2 \geq D^2 R^2(w)(\dot w,\dot w)$$
for every $w \in \hat\CW^>$ and $\dot w \in T_w \hat\CW^>$. From Proposition \ref{pro-4} we have
$$DR(w)\dot w = D \sigma_n(\pi_1(w))(D\pi_1(w)\dot w)={\rm Re}(D\pi_1(w)\dot w)_{nn},$$ 
so that $\|DR(w)\|=1$. On the other hand, assume that $\dot w\neq0$ and let $\gamma$ be a geodesic in $\hat\CW^>$, $\gamma(0)=w,\dot\gamma(0)=\dot w$. From Lemma \ref{lem-der1dies},
$$D^2R^2(w)(\dot w, \dot w) =(\sigma_n^2 \circ \pi_1\circ\gamma)''(0)=  $$
$$D^2 \sigma_n^2(\pi_1(w))(D\pi_1(w)\dot w,D\pi_1(w)\dot w)+2\sigma_nD \sigma_n(\pi_1(w))(\pi_1\circ\gamma)''(0)<$$
$$D^2 \sigma_n^2(\pi_1(w))(D\pi_1(w)(\dot w),D\pi_1(w)(\dot w)).$$
Thus, we have to prove that for $\dot y\in\K^{n\times(n+1)}$,
$$2 \left\|\dot y \right\|^2 \geq D^2\sigma_n^2(\pi_1(w))(\dot y,\dot y).$$
which is a consequence of our Proposition \ref{pro-4}. \end{proof}

\begin{corollary} \label{cor-4} The map $\al_2 : \CW^> \rightarrow \R$ given by $\al_2 (M,\zeta)=\|M\|_F^2/\si_n^2(M)$ is self-convex. 
\end{corollary}

\begin{proof} 
Consider the Riemannian submersion 
\[
p_1\times p_2:\S_1\times \S_2\longrightarrow\PKn\times\Pn,\ p_1\times p_2(M,\ze)=(p_1(M),p_2(\ze)).
\]
Note that $T_{(M,\ze)}\hat{\CW}^>$ contains the kernel of the derivative $D(p_1\times p_2)(M,\ze)$. Thus, the restriction $p_1\times p_2:\hat{\CW}^>\rightarrow \CW^>$, is also a Riemannian submersion. The corollary follows combining Proposition \ref{pro-RS} and Theorem \ref{th-2}.
\end{proof}

\section{Self-convexity of the distance from a submanifold of $\R^j$} 

Let $\CN$ be a $C^k$ submanifold without boundary $\CN \subset \R^j$, $k \ge 2$. Let us denote by
$$\rho(x) = d(x,\CN) = \inf_{y \in \CN} \left\| x-y \right\|$$
the distance from $\CN$ to $x \in \R^j$ (here $d(x,y)=\left\| x-y \right\|$ denotes the Euclidean distance). Let $\CU$ be the largest open set in $\R^j$ such that, for any $x \in \CU$, there is a unique closest point from $\CN$ to $x$. 
This point is denoted by $K(x)$ so that we have a map defined by
$$K : \CU \rightarrow \CN, \ \rho(x) = d(x,K(x)).$$
Classical properties of $\rho$ and $K$ are given in the following (see also Foote \cite{foo}, Li and Nirenberg \cite{li-nir}).

\begin{proposition} \label{pro-5} \begin{enumerate} \item $\rho$ is defined and $1-$Lipschitz on $\R^j$,
\item For any $x \in \CU$, $x-K(x)$ is a vector normal to $\CN$ at $K(x)$ i.e. $x-K(x) \in \left(T_{K(x)}\CN\right)^\perp$,
\item $K$ is $C^{k-1}$ on $\CU$,
\item $\rho^2$ is $C^{k}$ on $\CU$, $D\rho^2(x)\dot x =  2 \left\langle x - K(x), \dot x \right\rangle$ and $D^2\rho^2(x)(\dot x,\dot x)=2\|\x\|^2  - 2 \left\langle DK(x)\dot x , \dot x \right\rangle$
\item $\rho$ is $C^{k}$ on $\CU \setminus \CN$,
\item $\left\langle DK(x)\dot x , \dot x \right\rangle \ge 0$ for every $x \in \CU$ and $\dot x \in \R^j$.
\end{enumerate}
\end{proposition}

\proof \begin{enumerate} \item For any $x$ and $y$ one has $\rho(x) = d(x,K(x)) \le d(x,K(y)) \le d(x,y) + d(y,K(y)) = d(x,y) + \rho(y)$. Since $x$ and $y$ play a symmetric role we get $\left|\rho(x)-\rho(y)\right| \le d(x,y).$
\item This is the classical first order optimality condition in optimization. 
\item This classical result may be derived from the inverse function theorem applied to the canonical map defined on the normal bundle to $\CN$ 
$$\mbox{can} : \mathrm{N}\CN \rightarrow \R^j, \ \mbox{can}(y,n)=y+n,$$ 
for every $y \in \CN$ and $n \in N_y\CN = \left(T_{y}\CN\right)^\perp$. The normal bundle is a $C^{k-1}$ manifold, the canonical map is a $C^{k-1}$ diffeomorphism when restricted to the set $\{(y,n):y+tn\in\CU,\ \forall\ 0\leq t\leq 1\}$ and
$K(x)$ is easily given from $\mbox{can}^{-1}$.
\item The derivative of $\rho^2$ is equal to
$D\rho^2(x)\dot x = 2 \left\langle x - K(x), \dot x - DK(x)\dot x \right\rangle = 2 \left\langle x - K(x), \dot x \right\rangle$ because $DK(x)\dot x \in T_{K(x)} \CN$ and $x-K(x) \in \left(T_{K(x)}\CN\right)^\perp$. Thus $\nabla \rho^2(x) = 2(x-K(x))$ is $C^{k-1}$ on $\CU$ so that $\rho^2$ is $C^{k}$. The formula for $D^2\rho^2$ follows.
\item Obvious.
\item Let $x(t)$ be a curve in $\CU$ with $x(0) = x$. Let us denote $\frac{dx(t)}{dt} = \dot x(t)$, $\frac{d^2x(t)}{dt^2} = \ddot x(t)$, $y(t) = K(x(t))$, 
$\frac{dy(t)}{dt} = \dot y(t)$ and $\frac{d^2y(t)}{dt^2} = \ddot y(t)$. From the first order optimality condition we get
$$\left\langle x(t) - y(t), \dot y(t) \right\rangle = 0$$
whose derivative at $t=0$ is
$$\left\langle \dot x - \dot y, \dot y \right\rangle + \left\langle x - y, \ddot y \right\rangle = 0.$$
Thus
$$\left\langle DK(x)\dot x , \dot x \right\rangle = \left\langle \dot y, \dot x \right\rangle = \left\langle \dot y, \dot y \right\rangle - \left\langle x - y, \ddot y \right\rangle.$$
This last quantity is equal to $\left.\frac{1}{2}\frac{d^2}{dt^2}\left\|x-y(t)\right\|^2\right|_{t=0}$. It is nonnegative by the second order optimality condition.
\end{enumerate}
\qed

\vskip 3mm \noindent {\bf Proof of Theorem \ref{th-2} and Corollary \ref{cor-th-2}.} We are now able to prove our second main theorem. Let us denote $\al(x) = 1/{\rho(x)^2}.$ We shall prove that $\al$ is self-convex on $\CU$.
From proposition \ref{pro-3} it suffices to prove that, for every $\dot x \in \R^j$, 
\[
2\|\x\|^2\|D\rho(x)\|^2\geq D^2\rho^2(x)(\x,\x)
\]
or, according to Proposition \ref{pro-5}.4 and $\left\| D\rho \right\| = 1$, that
\[
2\|\x\|^2 \geq 2\|\x\|^2  - 2 \left\langle DK(x)\dot x , \dot x \right\rangle.
\]
This is obvious from Proposition \ref{pro-5}.4. 

Now we prove Corollary \ref{cor-th-2}. Let $\S_1(\R^j)$ be the sphere of radius $1$ in $\R^j$ and let $p_{\R^j}$ denote the canonical projection $p_{\R^j} : \R^j \rightarrow \P(\R^j)$. Note that the preimage of $\CN$ by $p_{\R^j}$ satisfies
$$d(y, p_{\R^j}^{-1}(\CN)) = d_\P (p_{\R^j}(y), \CN)  \| y \|.$$
As in the proof of Corollary \ref{cor-3b},
the mapping $1/{\rho(x)^2}$ is self-convex in the set $\S_1(\R^j) \cap p_{\R^j}^{-1}(\CU)$. Now, apply Proposition \ref{pro-RS} to the Riemannian submersion $p_{\R^j}$ to conclude the corollary.
\qed

\vskip 3mm \noindent {\bf Two examples.} 
\begin{example} \label{ex-circle} Take $\CU$ the unit disk in $\R^2$ and $\CN$ the unit circle. 
The corresponding function is given by
$$\al(x) = d(x, \CN)^{-2} = 1 / \left(1 - \left\| x \right\| \right)^2.$$
According to Theorem \ref{th-2}, the map $\log \al (x)$ is convex along the condition geodesics in 
$$ \CU \setminus \left\{\left(0,0\right)\right\} = \left\{x \in \R^2 \ : \ 0 < \left\| x \right\| < 1 \right\}.$$
This property also holds in $\CU$: a geodesic through the origin is a ray $x(t) = (-1 + e^t)(\cos \theta, \sin \theta)$ when 
$-\infty < t \leq 0$, and $x(t) = (1 - e^{-t})(\cos \theta, \sin \theta)$ when $0 \le t < \infty$ for some $\theta$. In that case
$$\log \al (x(t)) =  2 \left| t \right|$$
which is convex. 
\end{example}

\begin{example} \label{ex-two-points} Take $\CN \subset \R^2$ equal to the union of the two points $(-1,0)$ and $(1,0)$. 
In that case
$$\al(x)^{-1} = d(x,\CN)^{2} = \min \left((1+x_1)^2 + x_2^2, (1-x_1)^2 + x_2^2 \right).$$
It may be shown that for any $0 < a \leq 1/10$, the straight line segment is the only minimizing geodesic joining
the points $(0,-a)$ and $(0,a)$. Since $\log \al(0, t) = - \log(1+t^2)$
has a maximum at $t = 0$, $g(t)$, $-a \leq t \leq a$, cannot be log-convex. Here $\left\{ 0 \right\} \times \R$ is equal to the locus in $\R^2$ of points equally distant from the two nodes which is the set we avoid in Theorem \ref{th-2}.
\end{example}


\begin{thebibliography}{99}

\bibitem{bel}{\sc Beltr\'an C., and M. Shub}, 
{\em Complexity of B\'ezout's Theorem VII: Distances Estimates in the Condition Metric.} Foundations of Computational Mathematics, 9 (2009) 179-195.

\bibitem{boi}{\sc Boito P., and J.-P. Dedieu}, {\em The condition metric in the space of full rank rectangular matrices.} http://www.math.univ-toulouse.fr/~dedieu/Boito-Dedieu-future.pdf To appear in: SIAM J. Matrix Analysis.

\bibitem{cla}{\sc Clarke F. H.}, 
{\em Optimization and Nonsmooth Analysis.} Les Publications CRM (1989) ISBN 2-921120-01-1.

\bibitem{dem}{\sc Demmel J. W.}, 
{\em The probability that a Numerical Problem is Difficult.} Mathematics of Computation, 50 (1988) 449-480.

\bibitem{foo}{\sc Foote R.}, 
{\em Regularity of the distance function}, Proceedings of the AMS, 92 (1984) pp 153-155.

\bibitem{gal}{\sc Gallot S., D. Hulin and J. Lafontaine}, 
{\em Riemannian Geometry}, Springer (2004) ISBN 9780387524016.

\bibitem{li-nir}{\sc Li Y. and L. Nirenberg}, 
{\em Regularity of the distance function to the boundary}, Rendiconti Accad. Naz. delle Sc. 123 (2005) pp 257-264.

\bibitem{shu}{\sc Shub M.}, 
{\em Complexity of B\'ezout's Theorem VI: Geodesics in the Condition Metric.} Foundations of Computational Mathematics, 
9 (2009) 171-178.

\bibitem{udr}{\sc Udriste, C.}, {\em Convex Functions and Optimization
Methods on Riemannian Manifolds}, Kluwer (1994) ISBN 0-7923-3002-1.

\end{thebibliography}
\end{document}